\documentclass{amsart} 
\author{Masaki Kameko}
\address{Department of Mathematical Sciences,
College of Systems Engineering and Science, 
Shibaura Institute of Technology,
307 Minuma-ku Fukasaku, Saitama-City 337-8570, Japan}
\email{kameko@shibaura-it.ac.jp}
\thanks{JSPS KAKENHI Grant Number 25K07013 supported this work.}

\usepackage{amssymb,amscd, pb-diagram}
\usepackage[initials]{amsrefs}
\newtheorem{theorem}{Theorem}[section]
\newtheorem{proposition}[theorem]{Proposition}
\newtheorem{lemma}[theorem]{Lemma}

\theoremstyle{definition}

\begin{document}
\title{The $2$-local homotopy types of $G_2$-gauge groups} 
\begin{abstract} We determine the $2$-local homotopy types of $G_2$-gauge groups over $S^4$. \end{abstract}
\maketitle


\newcommand{\ktt}{Kishimoto, Theriault and Tsutaya}
\newcommand{\oshima}{\=Oshima}

\section{Introduction}\label{sec:1}

Let $G$ be a compact Lie group and let $P$ be a principal $G$-bundle over the $4$-dimensional sphere $S^4$. The $G$-gauge group over the base space $S^4$ is the topological group of $G$-bundle automorphisms of $P$. Suppose that $G$ is simply-connected and simple. Its classifying space $BG$ is $3$-connected and $\pi_4(BG)\cong \mathbb{Z}$. We denote a map and its homotopy class by the same symbol. Then, the principal $G$-bundle $P$ is classified by a map $k\colon S^4 \to BG$ in $\pi_4(BG)$ and there are infinitely many isomorphism classes of principal $G$-bundles over $S^4$ since $\pi_4(BG) \cong \mathbb{Z}$ is infinite. In \cite{kono-1991}, Kono classified the homotopy types of $SU(2)$-gauge groups over $S^4$ and showed that there are six homotopy types of these groups.

Since then, many classification results on the homotopy types of gauge groups of low-rank simple Lie groups have been obtained. Among compact Lie groups, $SU(2)$ is the rank $1$ simply-connected simple Lie group. There are three simply-connected simple compact Lie groups of rank $2$: $SU(3)$, $Sp(2)$, and $G_2$. For $G=SU(3)$, Hamanaka and Kono \cite{hamanaka-kono-2006} proved that $\mathcal{G}_k\simeq \mathcal{G}_{k'}$ if and only if $(k, 24)=(k', 24)$, where $(m,n)$ is the greatest common divisor of $m$ and $n$. For $G=Sp(2)$, Theriault \cite{theriault-2010} showed $\mathcal{G}_k\simeq_{(p)} \mathcal{G}_{k'}$ at any prime $p$ if and only if $(k, 40)=(k', 40)$, where $\simeq_{(p)}$ means $p$-local homotopy equivalence. Then, for $G=G_2$, Kishimoto, Theriault and Tsutaya \cite{ktt-2017} showed that $\mathcal{G}_k\simeq_{(p)} \mathcal{G}_{k'}$ at any prime $p$ if $(k, 168)=(k', 168)$ and if $\mathcal{G}_k \simeq_{(p)} \mathcal{G}_{k'}$ at any prime $p$, then $(k,84)=(k', 84)$. The discrepancy between $(k, 168)=(k',168)$ and $(k, 84)=(k', 84)$ has remained as an open problem. 

In this paper, we complete the classification for rank $2$ simple Lie groups up to $p$-local homotopy equivalence at any prime $p$ by proving the following.


\begin{theorem}\label{theorem:1.1}
Let $\mathcal{G}_k$ be the gauge group of a principal $G_2$-bundle over $S^4$ whose classifying map is $k\in \mathbb{Z}\cong \pi_{4}(BG_2)$. $\mathcal{G}_k\simeq_{(p)} \mathcal{G}_{k'}$ at any prime $p$ if and only if $(k, 168)=(k', 168)$.
\end{theorem}


Theorem~\ref{theorem:1.1} is equivalent to the following local form. 

\begin{proposition}\label{proposition:1.2}
Let $\mathcal{G}_k$ be the gauge group of a principal $G_2$-bundle over $S^4$ whose classifying map is $k\in \mathbb{Z}\cong \pi_{4}(BG_2)$.
Then, we have 
\begin{itemize}
\item[{\rm (1)}] $\mathcal{G}_k\simeq_{(2)} \mathcal{G}_{k'}$  if and only if $(k, 8)=(k', 8)$, 
\item[{\rm (2)}] 
$\mathcal{G}_k\simeq_{(3)} \mathcal{G}_{k'}$  if and only if $(k, 3)=(k', 3)$, 
\item[{\rm (3)}] $\mathcal{G}_k\simeq_{(7)} \mathcal{G}_{k'}$  if and only if $(k, 7)=(k', 7)$
and 
\item[{\rm (4)}] $\mathcal{G}_k\simeq_{(p)} \mathcal{G}_{k'} \simeq _{(p)} G_2 \times \Omega_0^4 G_2$ for $p\not=2,3,7$.
\end{itemize}
\end{proposition}

Proposition~\ref{proposition:1.2} (2), (3), (4) and the ``if" part of (1) were established in \cite{ktt-2017}*{Propositions 3.2 and 4.6 and Theorem 1.2}. In this paper, we prove the ``only if" part of Proposition~\ref{proposition:1.2} (1) and complete the proof of Theorem~\ref{theorem:1.1} assuming the results of \cite{ktt-2017}.


First, we recall the basic method for determining the homotopy types of $G$-gauge groups over $S^4$, which relates the homotopy types of gauge groups to Samelson products. We begin with the following homotopy fiber sequence of mapping spaces.
\[
\Omega {\mathrm{Map}} (S^4, BG)_k  \to G \to  {\mathrm{Map}}_* (S^4, BG)_k \to {\mathrm{Map}} (S^4, BG)_k \stackrel{\mathrm{ev}}{\longrightarrow} BG,
\]
where $\mathrm{Map} (S^4, BG)_k$ is the connected component of the space of continuous maps containing the map $k\colon S^4\to BG$, $\mathrm{Map}_* (S^4, BG)_k$ is its subspace consisting of based maps,  and $\mathrm{ev}$ is the evaluation map.
There is a homotopy equivalence  $\Omega_0^3 G\simeq \mathrm{Map}_{*}(S^4, BG)_0 \simeq \mathrm{Map}_{*}(S^4, BG)_k$ and Gottlieb \cite{gottlieb-1972} showed that the classifying space of the gauge group $B\mathcal{G}_k$ is homotopy equivalent to the mapping space $ \mathop{\mathrm{Map}} (S^4, BG)_k$. Therefore, we have the following homotopy fiber sequence.
\[
\mathcal{G}_k  \stackrel{h_k}{\longrightarrow} G \stackrel{\partial_k}{\longrightarrow} \Omega_0^3 G \to B\mathcal{G}_k \to BG.
\]
Thus, the $G$-gauge group $\mathcal{G}_k$ is homotopy equivalent to the homotopy fiber of the map $\partial_k$. Let $i_3\colon S^3 \to G$ be the inclusion map of the bottom cell. Then, Lang \cite{lang-1973} proved that the map $\partial_k$ is the triple adjoint of the Samelson product $\langle k\cdot i_3, 1\rangle$ where $1$ is the identity map of $G$. By the linearity of the Samelson product, we have ${\langle k\cdot i_3, 1\rangle} \simeq {k \cdot \langle i_3, 1\rangle}$. Furthermore, in \cite{theriault-2010}, Theriault showed that $\mathcal{G}_k$ is $p$-locally homotopy equivalent to $\mathcal{G}_{k'}$ at all primes $p$ if $(k, m)=(k',m)$ where $m$ is the order of the Samelson product $\langle i_3, 1\rangle$. 


For an abelian group $A$, we denote by $A_{(p)}$ its localization at the prime $p$. If $A$ is finite, then $A_{(p)}$ is isomorphic to the $p$-primary subgroup of $A$. We denote by ${}_2 A$ the $2$-primary subgroup of $A$. The first new ingredient is the following proposition.

\begin{proposition}\label{proposition:1.3}  The order of the Samelson product $\langle i_3, 1\rangle$ in $[\Sigma^3 G_2, G_2]_{(2)}$ is $8$.
\end{proposition}


As for the order of the Samelson product $\langle i_3, 1\rangle$ in $[\Sigma^3 G_2, G_2]$, by Proposition~\ref{proposition:1.3}, together with \cite{ktt-2017}*{Lemma 3.1, Proposition 4.5 and Corollary 6.6}, we have the following theorem.

\begin{theorem}\label{theorem:1.4} 
The order of the Samelson product $\langle i_3, 1\rangle$ in $[\Sigma^3 G_2, G_2]$ is $168$.
\end{theorem}

Next, we recall some properties of a space $V:=G_2/SU(2)$, which played an important role in \cite{ktt-2017} and is crucial in this paper. Let $q\colon G_2\to V$ be the obvious projection map. By \cite{ktt-2017}*{Lemmas 7.1 and 7.2}, the Samelson product $3 \cdot \langle i_3, 1\rangle\colon \Sigma^3 G_2\to G_2$ factors through $\Sigma^3 q \colon \Sigma^3 G_2 \to \Sigma^3 V$. 
Let us write this factorization as follows.
 \[
3\cdot \langle i_3, 1\rangle\simeq \psi\circ \Sigma^3 q.
\]
Let $P^{n+1}(2)$ be the mapping cone of the degree $2$ map $\times 2 \colon S^n\to S^n$.
The $6$-skeleton of $V=S^5\cup e^6\cup e^{11}$ is $P^6(2)$ and, by \cite{ktt-2017}*{Lemma 7.4}, there is a homotopy equivalence $\Sigma^2 V\simeq P^8(2) \vee S^{13}$.  Let $
p_{11}\colon V \to S^{11}
$
 be the pinch map to the top cell, collapsing the $6$-skeleton to the base point. We denote by 
 $
 i_6\colon P^6(2) \to V
 $
  the inclusion map. 
The above homotopy equivalence $\Sigma^2 V\simeq P^8(2)\vee S^{13}$ provides suspension maps $s'\colon \Sigma^3 V \to P^9(2)$ and $s''\colon S^{14}\to \Sigma^3V$ of splitting maps associated with $\Sigma^2 V$ such that the identity map of $\Sigma^3 V$ is homotopic to 
  \[
  \Sigma^3 i_6 \circ s' + s'' \circ \Sigma^3 p_{11}.
  \] 
In this paper, we prove the following proposition.


\begin{proposition} \label{proposition:1.5} $84\cdot \psi\circ s''\circ \Sigma^3 (p_{11}\circ q)$ has order $2$ in $[ \Sigma^3 G_2, G_2]$. \end{proposition}


\begin{proof}[Proof of Proposition~\ref{proposition:1.3}] We consider the following decomposition in $[\Sigma^3 G_2, G_2]$. \[ 3 \cdot \langle i_3, 1\rangle\simeq \psi \circ \Sigma^3 i_6\circ s' \circ \Sigma^3 q +\psi \circ s'' \circ \Sigma^3 (p_{11}\circ q). \]  
By \cite{ktt-2017}*{Lemma 5.2 (a) and (b)}, the order of the identity map $1_P$ of $P^9(2)$ is $4$. Since $s'$ is a suspension, we have
$4 \cdot s' = 1_{P} \circ (4  \cdot s' ) \simeq (4 \cdot 1_{P}) \circ s' \simeq 0$. Therefore, 
we have $84\cdot 3 \cdot \langle i_3, 1\rangle \simeq 84 \cdot \psi \circ s'' \circ \Sigma^3 (p_{11}\circ q)$. 
Since multiplication by any odd integer on $[\Sigma^3 G_2, G_2]_{(2)}$ is injective, the order of 
$4 \cdot \langle i_3, 1\rangle$ is the same as the order of $84\cdot \psi\circ s''\circ \Sigma^3 (p_{11}\circ q)$. Thus, by Proposition~\ref{proposition:1.5}, $\langle i_3, 1\rangle$ has order $8$ in $[\Sigma^3 G_2, G_2]_{(2)}$.
\end{proof}
 
Proposition~\ref{proposition:1.3} itself does not complete the classification of $2$-local homotopy types of gauge groups. To complete the classification, we need to show that $\mathcal{G}_0$, $\mathcal{G}_1$, $\mathcal{G}_2$, $\mathcal{G}_4$ have distinct $2$-local homotopy types. We do this by making use of a result obtained in the proof of Proposition~\ref{proposition:1.5}.

In what follows, we use the symbol $p$ to denote the map $p\colon G_2\to S^6\simeq G_2/SU(3)$ unless otherwise stated explicitly.
When we wish to indicate the generator of a cyclic group, we write $\mathbb{Z}/m\{ x\}$ for the cyclic group of order $m$ generated by $x$. Similarly, we write $\mathbb{Z}/m\{x, y, \dots, z \}$ for $\mathbb{Z}/m\{x\}\oplus \mathbb{Z}/m\{y\}\oplus \cdots\oplus \mathbb{Z}/m\{z\}$. For a map $f\colon X\to Y$, we denote its mapping cone by $j(f)\colon Y \to C(f)$ and let $\pi(f)\colon C(f)\to \Sigma X$ be the pinch map collapsing the subspace $Y$ of $C(f)$ to the base point. 

This paper is organized as follows. In Section \ref{sec:2}, we collect some facts about homotopy groups of $G_2$ and $S^6$ and reduce the proof of Proposition~\ref{proposition:1.5} to a problem on $\pi_{14}(S^6)$. In Section \ref{sec:3}, we prove Proposition~\ref{proposition:1.5} assuming Propositions~\ref{proposition:3.1} and \ref{proposition:3.10}.
In Section~\ref{sec:4}, we recall relative Samelson products and prove Proposition~\ref{proposition:3.1}.
In Section~\ref{sec:5}, we prove Proposition~\ref{proposition:3.10}. In Section~\ref{sec:6}, we close this paper by proving Proposition~\ref{proposition:1.2}.

{\bf Acknowledgments} The author would like to thank the anonymous referee for pointing out a serious error in the initial version of this manuscript and for carefully reading the revised drafts.
Thanks are also due to Daisuke Kishimoto for helpful conversations, and to Dang Vo Phuc \cite{phuc} for identifying several errors in a subsequent draft.


\section{Homotopy groups}\label{sec:2}

For $3\leq s \leq 14$, the homotopy groups $\pi_s(G_2)$ were computed by Mimura \cite{mimura-1967}*{\S1}.
\begin{align*}
\pi_{3}(G_2)&\cong \mathbb{Z}, & \pi_{4}(G_2)&\cong \{0\}, & \pi_{5}(G_2)&\cong \{0\}, & \pi_6(G_2) &\cong \mathbb{Z}/3, 
\\
\pi_{7}(G_2)&\cong \{0\}, & \pi_8(G_2)&\cong \mathbb{Z}/2, & \pi_9(G_2)&\cong \mathbb{Z}/6, & \pi_{10}(G_2)&\cong \{ 0\},
\\
 \pi_{11}(G_2)&\cong \mathbb{Z}\oplus \mathbb{Z}/2, &\pi_{12}(G_2)&\cong \{0\}, & \pi_{13}(G_2)&\cong \{0\},  & \pi_{14}(G_2)&\cong \mathbb{Z}/168\oplus \mathbb{Z}/2.
\end{align*}
For the same range, the homotopy groups of $S^6$ are given as follows.
\begin{align*}
\pi_{3}(S^6)&\cong \{0\}, & \pi_{4}(S^6)&\cong \{0\}, & \pi_{5}(S^6)&\cong \{0\}, & \pi_6(S^6) &\cong \mathbb{Z}, 
\\
\pi_{7}(S^6)&\cong \mathbb{Z}/2, & \pi_8(S^6)&\cong \mathbb{Z}/2, & \pi_9(S^6)&\cong \mathbb{Z}/24, & \pi_{10}(S^6)&\cong \{ 0\},
\\
 \pi_{11}(S^6)&\cong \mathbb{Z}, &\pi_{12}(S^6)&\cong \mathbb{Z}/2, & \pi_{13}(S^6)&\cong \mathbb{Z}/60,  & \pi_{14}(S^6)&\cong \mathbb{Z}/24\oplus \mathbb{Z}/2.
\end{align*}

We collect some facts about the homotopy groups of spheres from Toda’s book \cite{toda-1962}.
Let us denote by $\iota_n$ the identity map of the $n$-dimensional sphere $S^n$. We use the notation of Toda's book \cite{toda-1962} to express elements of homotopy groups of spheres, except that we denote the suspension by $\Sigma$ instead of $E$. $H$ and $\Delta$ are the same as those in Toda's book. For a map $\alpha \in \pi_s(S^n)$ and an integer $a$, we define $a\alpha =a\cdot \alpha$ as a map homotopic to $\alpha \circ a\iota_s$. 


\begin{proposition}\label{proposition:2.1}
The following holds:\\
{\rm (1)} $\pi_{11}(S^6)\cong \mathbb{Z}\{ \Delta(\iota_{13})\}$, \\
{\rm (2)} ${}_2\pi_{14}(S^6)\cong \mathbb{Z}/8\{ \bar{\nu}_6\} \oplus \mathbb{Z}/2\{ \epsilon_6\}$,
\\
{\rm (3)} ${}_2\pi_{17}(S^6)\cong \mathbb{Z}/8\{ \zeta_6\} \oplus \mathbb{Z}/4\{ \bar{\nu}_6 \circ \nu_{14}\} $,
\\
{\rm (4)} $\Delta(\iota_{13})\circ \nu_{11}\simeq \pm 2\bar{\nu}_6$,
\\
{\rm (5)} $\nu_6\circ \bar{\nu}_9\simeq 2 \bar{\nu}_{6}\circ \nu_{14}$,
\\
{\rm (6)} $\nu_6\circ \epsilon_9\simeq 2 \bar{\nu}_{6}\circ \nu_{14}$,
\\
{\rm (7)} $\epsilon_6\circ \nu_{14}\simeq 0$, 
\\
{\rm (8)} $\eta_6^2 \circ \mu_8\simeq 4\zeta_6$, 
\\
{\rm (9)} $\eta_6 \circ \nu_7 \circ \sigma_{10}\simeq 0$.
\end{proposition}

\begin{proof} We refer the reader to Toda's book for the proof.
(1), (2), (3) are in \cite{toda-1962}*{Proposition 5.9, Theorem 7.1, Theorem 7.4}.
(4) is in \cite{toda-1962}*{page 53, line -4}.
(5) and (6) are immediate from \cite{toda-1962}*{(7.17), (7.18)}.
For (7), recall that
${}_2 \pi_{6}(S^3)\cong \mathbb{Z}/4\{ \nu'\}$ and $\nu' \circ \epsilon_6 \simeq \epsilon_3 \circ \nu_{11}$ in \cite{toda-1962}*{(7.12)}.
Since
$\Sigma(2\nu_4 -\Sigma \nu')\simeq 0$,
we have
$\Sigma^3 \nu' \simeq 2\nu_6$. 
Therefore, we obtain
$\epsilon_6\circ \nu_{14}\simeq \Sigma^3 \nu' \circ \epsilon_9\simeq 2\nu_6 \circ \epsilon_9\simeq \nu_6\circ 2\epsilon_9\simeq 0$.
(8) follows from $4\zeta_5\simeq \eta_5^2 \circ \mu_7$ in  \cite{toda-1962}*{page 69, line 12}.
(9) is immediate from the fact that $\eta_n\circ \nu_{n+1}\simeq 0$ for $n\geq 5$ in \cite{toda-1962}*{(5.9)}.
\end{proof}

It is easy to see that if $\alpha\in \pi_s(S^n)$ is a suspension, then  $(a\iota_n) \circ \alpha\simeq \alpha\circ (a \iota_s)$. 
However, in general, if $\alpha $ is not a suspension,  $(a \iota_n) \circ \alpha$ may not be homotopic to $\alpha\circ (a\iota_s)$. The following proposition records two such examples needed below.


\begin{proposition}\label{proposition:2.2}
We have
$(2 \iota_{6})\circ \Delta(\iota_{13})\simeq 4 \Delta(\iota_{13})$ in $\pi_{11}(S^6)$ and 
$(2\iota_{6} )\circ \bar{\nu}_6\simeq 4 \bar{\nu}_6$ in $\pi_{14}(S^6)$.
\end{proposition}

\begin{proof}
The Hopf invariants
$H\colon \pi_s(S^{n+1})\to \pi_{s}(S^{2n+1})$ are defined as homomorphisms of homotopy groups.
By \cite{toda-1962}*{Proposition~{2.2}}, for $\gamma\in \pi_{n}(S^m)$ and $\alpha\in \pi_{s}(S^{n+1})$, we have
\[
H(\Sigma \gamma \circ \alpha)\simeq \Sigma (\gamma \wedge \gamma)\circ H(\alpha).
\]
Moreover, we have a long exact sequence
\[
\cdots \to \pi_{s-1}(S^n)_{(2)}\stackrel{\Sigma}{\longrightarrow} \pi_{s}(S^{n+1})_{(2)} \stackrel{H}{\longrightarrow} \pi_{s}(S^{2n+1})_{(2)} \stackrel{\Delta}{\longrightarrow} \pi_{s-2}(S^n)_{(2)} \to \cdots 
\]

We consider the case $n=5$ and $\Sigma \gamma=2 \iota_6$, first with $\alpha=\Delta(\iota_{13})$ and then with $\bar{\nu}_6$.

First, we deal with the case $\alpha=\Delta(\iota_{13})$. Since $\pi_{10}(S^5)$ is a finite group and $\pi_{11}(S^6)$ is torsion-free, the suspension map 
$\Sigma \colon \pi_{10}(S^5)\to \pi_{11}(S^6)$ is zero. 
Since $\pi_{10}(S^{5})_{(2)} \stackrel{\Sigma}{\longrightarrow} \pi_{11}(S^{6})_{(2)} \stackrel{H}{\longrightarrow}  \pi_{11}(S^{11})_{(2)}$ 
is exact, $H\colon \pi_{11}(S^6)_{(2)}\to \pi_{11}(S^{11})_{(2)}$ is injective. Hence, $H\colon \pi_{11}(S^6)\to \pi_{11}(S^{11})$ is also injective and we have $ H(\Delta(\iota_{13}))\simeq a \iota_{11}$.
Then
$H((2\iota_6) \circ \Delta(\iota_{13}))\simeq \Sigma(2\iota_5\wedge 2\iota_5) \circ H(\Delta(\iota_{13}))\simeq  4 \iota_{11} \circ a \iota_{11}\simeq (4a) \iota_{11}\simeq H(4\Delta(\iota_{13}))$. Hence we obtain $(2\iota_{6})\circ \Delta(\iota_{13})\simeq 4 \Delta(\iota_{13})$.

Next, we deal with the case $\alpha=\bar{\nu}_6$.  
By \cite{toda-1962}*{Lemma~{6.2}}, there is an integer $b$ such that $H(\bar{\nu}_6)\simeq (1+2b)\cdot \nu_{11}$.
On the one hand, we have
$H((2\iota_6)\circ \bar{\nu}_6)\simeq \Sigma (2 \iota_5 \wedge 2\iota_5) \circ H(\bar{\nu}_6)\simeq 4\nu_{11}$. On the other hand, we have $H(4\bar{\nu}_6)\simeq 4 \nu_{11}$.
Hence, we have $H((2\iota_6)\circ \bar{\nu}_6-4\bar{\nu}_6)\simeq 0$. Since
$\pi_{13}(S^{5})_{(2)}\stackrel{\Sigma}{\longrightarrow} \pi_{14}(S^{6})_{(2)} \stackrel{H}{\longrightarrow} \pi_{14}(S^{11})_{(2)}$ is exact and, since $\pi_{13}(S^{5})_{(2)}\cong \mathbb{Z}/2\{ \epsilon_5\}$ by \cite{toda-1962}*{Theorem 7.1}, 
there is an integer $c$ such that
$(2\iota_6)\circ \bar{\nu}_6-4\bar{\nu}_6\simeq c\epsilon_6$.
The stabilization map sends $(2\iota_6)\circ \bar{\nu}_6-4\bar{\nu}_6$ and $c \epsilon_6$ to $2\bar{\nu}-4 \bar{\nu}\simeq 0$ and $c \epsilon$  in the stable homotopy classes $\{ S^{14}, S^6\}\cong \mathbb{Z}/2\{ \bar{\nu}, \epsilon\}$, respectively. Hence, $c=0$. Thus, we have $(2\iota_6)\circ \bar{\nu}_6-4\bar{\nu}_6\simeq 0$ in $\pi_{14}(S^6)$.
\end{proof}

Next, we recall some generators of the homotopy groups of the Lie group $G_2$.
Mimura computed the homotopy groups of Lie groups of low rank using the homotopy groups of spheres. 
For instance, fiber sequences $S^3 \to SU(3) \to S^5$, $SU(3) \stackrel{i}{\longrightarrow} G_2 \stackrel{p}{\longrightarrow} S^6$, where $S^3=SU(2)$, $S^5=SU(3)/SU(2)$, $S^6=G_2/SU(3)$, give us exact sequences involving homotopy groups of spheres and homotopy groups of $SU(3)$ and $G_2$. 
The following is in \cite{mimura-1967}*{Theorem 6.1}.


\begin{proposition} \label{proposition:2.3}
Let $i\colon SU(3)\to G_2$ be the inclusion map and 
$p\colon G_2\to S^6=G_2/SU(3)$ the projection map.
By $\langle \alpha \rangle$, we denote an element in $\pi_s(G_2)$ such that $p_*(\langle \alpha\rangle)\simeq \alpha$. 
The following holds.
\\
{\rm (1)} $\pi_{11}(G_{2})\cong \mathbb{Z}\{ \langle 2 \Delta(\iota_{13})\rangle\} \oplus  \mathbb{Z}/2\{ i_{*}([\nu_{5}^{2}])\}$, 
\\
{\rm (2)} ${}_2\pi_{14}(G_{2})\cong \mathbb{Z}/8\{ \langle \bar{\nu}_{6}+\epsilon_{6}\rangle\} \oplus \mathbb{Z}/2\{ i_{*}([\nu_{5}^{3}])\}$, and 
\\
{\rm (3)} $\pi_{17}(G_2)= {}_2\pi_{17}(G_2)\cong \mathbb{Z}/8\{ \langle \bar{\nu}_6 \circ \nu_{14} \rangle\}\oplus \mathbb{Z}/2\{ \langle 4\zeta_6\rangle \}$.
\end{proposition}

We close this section by proving the following propositions. With these propositions, we may reduce the computation of the Samelson product $\langle i_3, 1\rangle$ to a problem on the homotopy group $\pi_{14}(S^6)$.


\begin{proposition}\label{proposition:2.4}
The induced homomorphism $\Sigma^3 (p_{11}\circ q)^*\colon \pi_{14}(G_2)\to [\Sigma^3 G_2, G_2]$ is injective.
\end{proposition}

\begin{proof}
Let $G_{2,11}$ be the $11$-skeleton of $G_2$, $i'\colon G_{2,11}\to G_2$ the inclusion map and $q'\colon G_{2,11} \to V$ the restriction of $q$ to $G_{2,11}$, so that $q'={q\circ i'}$.
Let $j(q')\colon V \to C(q')$ be the mapping cone of $q'$. 
There is a cofiber sequence 
$
S^{12} \to P^{13}(2)\to S^{13}.
$
It induces an exact sequence
\[
\pi_{12}(G_2) \leftarrow [P^{13}(2), G_2] \leftarrow \pi_{13}(G_2).
\]
Since $\pi_{13}(G_2)\cong \pi_{12}(G_2)\cong \{0\}$, we have 
$[ P^{13}(2), G_2]\cong \{0\}$.
Furthermore, since $\pi_{7}(G_2)\cong \{0\}$, and since we have a cofiber sequence 
\[
S^7 \to \Sigma^3 C(q')\to P^{13}(2),
\]
we have $[\Sigma^3 C(q'),G_2]\cong \{0\}$.
Hence, the induced homomorphism
$
(\Sigma^3 q')^*\colon [\Sigma^3 V, G_2] \to [\Sigma^3 G_{2,11} , G_2]$ is injective. Therefore, $(\Sigma^3 q)^*\colon [\Sigma^3 V, G_2] \to [\Sigma^3 G_{2} , G_2]$ is also injective.
Furthermore, since the induced homomorphism $(\Sigma^3 p_{11})^*$ has the  splitting $s''^*$, $(\Sigma^3 p_{11})^*$ is injective. Therefore, we obtain the proposition.
\end{proof}


\begin{proposition} \label{proposition:2.5}
If there is an odd integer $a''$ such that 
${21\cdot (p\circ \psi\circ s'' )}\simeq a'' (\bar{\nu}_6+\epsilon_6)$, then ${84 \cdot (\psi\circ s''\circ \Sigma^3 (p_{11}\circ q))}\simeq 4\cdot \langle \bar{\nu}_6+\epsilon_6\rangle\circ \Sigma^3 (p_{11}\circ q)$ and it has order $2$.
\end{proposition}

\begin{proof}
Since $\pi_{14}(G_2)\cong \mathbb{Z}/168\oplus \mathbb{Z}/2$, the homotopy class $21\cdot (\psi \circ s'')$ lies in the $2$-primary component of $\pi_{14}(G_2)$.  On the one hand, by Proposition~\ref{proposition:2.3}, there are integers $a$, $b$ such that 
$21 \cdot {(\psi\circ s'')} \simeq a \cdot \langle \bar{\nu}_6+\epsilon_6\rangle+ b \cdot i_*([\nu_5^3])$. So, we have
${84 \cdot (\psi \circ s'')} \simeq 4 a \cdot \langle \bar{\nu}_6+\epsilon_6\rangle$.
Hence, we have $84\cdot {(p\circ \psi \circ s'')} \simeq 4a \cdot (\bar{\nu}_6+\epsilon_6)\simeq 4a \cdot \bar{\nu}_6$.
On the other hand, by the assumption that there is an odd integer $a''$ such that $21\cdot {(p\circ \psi\circ s'')} \simeq a''\cdot (\bar{\nu}_6+\epsilon_6)$, we have $84 \cdot {(p\circ \psi\circ s'')} \simeq {4 a'' \cdot (\bar{\nu}_6+\epsilon_6)} \simeq {4\cdot \bar{\nu}_6}$. 
Hence, we have $a \equiv 1 \mod (2)$ and $84\cdot {(\psi\circ s'')}\simeq 4 \cdot \langle \bar{\nu}_6+\epsilon_6\rangle$. By Proposition~\ref{proposition:2.4}, we obtain the proposition.
\end{proof}


\section{Proof of Proposition~\ref{proposition:1.5}}\label{sec:3}

In this section, we prove Proposition~\ref{proposition:1.5} assuming Propositions~\ref{proposition:3.1} and \ref{proposition:3.10}. 
We begin by establishing the notation and terminology used in this section. 
For maps $f_1\colon Y_1 \to Z$ and $f_2\colon Y_2 \to Z$, we write $f_1\dot{+}f_2\colon Y_1 \vee Y_2 \to Z$ for the map $f_{Z} \circ (f_1 \vee f_2)\colon Y_1\vee Y_2\to Z$ where $f_Z\colon Z \vee Z \to Z$ is the fold map. 
For maps $g_1\colon \Sigma X \to  Y_1$ and $g_2 \colon\Sigma X  \to  Y_2$, we write $g_1\dot{\vee} g_2\colon \Sigma X \to Y_1 \vee Y_2$ for $(g_1\vee g_2) \circ p_{\Sigma X}$ where $p_{\Sigma X}\colon \Sigma X \to \Sigma X \vee \Sigma X$ is the pinch map. 
If $Y_1=Y_2=\Sigma Y$, we define $f_1+f_2$ by $f_Z\circ (f_1\dot{+} f_2)\circ p_{\Sigma Y}$, so that 
\[
(f_1 \dot{+}f_2)\circ (g_1\dot{\vee} g_2)=f_1\circ g_1 + f_2 \circ g_2=f_Z \circ ( f_1 \circ g_1\vee f_2 \circ g_2)\circ p_{\Sigma X}.
\]
We consider the following commutative diagram, where $S^{3}=SU(2)$, $S^5=SU(3)/SU(2)$, $S^6=G_2/SU(3)$ and all vertical and horizontal sequences are fiber sequences. 
\[
\begin{diagram}
\node{S^{3}} \arrow{e,t}{i'_3} \arrow{s,l}{=}\node{SU(3)} \arrow{s,r}{i}\arrow{e} \node{S^5} \arrow{s,r}{i_5} \\
\node{S^{3}} \arrow{e,t}{i_{3}} \arrow{s}  \node{G_2} \arrow{e,t}{q}\arrow{s,r}{p}  \node{V} \arrow{s,r}{p_{6}} \\
\node{*} \arrow{e} \node{S^6} \arrow{e,t}{=} \node{S^6}
\end{diagram}
\]
As in Section~\ref{sec:1}, 
let $i_6\colon P^6(2)\to V$ be the inclusion map of the $6$-skeleton of $V=S^5\cup e^6\cup e^{11}$.
Let $p_{11}\colon V \to S^{11}$ be the pinch map collapsing $P^6(2)$ to the base point. 
Let $s'\colon P^9(2)\to \Sigma^3 V$, $s''\colon S^{14}\to \Sigma^3 V$ be suspension maps such that maps $s'\circ \Sigma^3 i_6$, $\Sigma^3 p_{11}\circ s''$ and $\Sigma^3 i_6 \circ s'+s''\circ \Sigma^3 p_{11}$ are homotopic to
the identity maps of $P^9(2)$, $S^{14}$ and $\Sigma^3 V$, respectively.
The map $\psi\colon \Sigma^3 V \to S^6$ is a map defined by $\psi \circ \Sigma^3 q \simeq 3\cdot \langle  i_3, 1\rangle$ where $1\colon G_2\to G_2$ is the identity map of $G_2$.


We assume the following Proposition~\ref{proposition:3.1} for the time being. We prove Proposition~\ref{proposition:3.1} in the next section.


\begin{proposition}\label{proposition:3.1} We have  
$
p\circ \langle i_3, \langle \bar{\nu}_6+\epsilon_6 \rangle\rangle  \simeq 2\bar{\nu}_6\circ \nu_{14}.
$
\end{proposition}

The following proposition is an immediate consequence of Proposition~\ref{proposition:3.1}.


\begin{proposition}\label{proposition:3.2}
We have
$p\circ \psi \circ \Sigma^3 (q\circ \langle \bar{\nu}_6+\epsilon_6 \rangle)\simeq 2\bar{\nu}_6\circ \nu_{14}$.

\end{proposition}

\begin{proof}
By the definition of $\psi$, we have 
\[
p\circ \psi \circ \Sigma^3 (q\circ \langle \alpha \rangle) \simeq  3 \cdot p\circ \langle i_3, \langle \alpha \rangle\rangle.
\]
For $\alpha=\bar{\nu}_6+\epsilon_6$, since $\bar{\nu}_6\circ \nu_{14}$ has order $4$, we have
$3 \cdot 2\bar{\nu}_6\circ \nu_{14}\simeq 2\bar{\nu}_6\circ \nu_{14}$.
\end{proof}


The goal of this section is to show that $a''$ in the following proposition is odd.

\begin{proposition} \label{proposition:3.3}
There is an integer $a''$ such that 
${21\cdot p \circ \psi \circ s''} \simeq a'' \cdot (\bar{\nu}_6+\epsilon_6)$ in $\pi_{14}(S^6)$.
\end{proposition}

\begin{proof}
Since $\pi_{14}(G_2)\cong \mathbb{Z}/168 \oplus \mathbb{Z}/2$, the subgroup ${21 \cdot \pi_{14}(G_2)}$ is the $2$-primary component of $\pi_{14}(G_2)$. The $2$-primary component of $\pi_{14}(G_2)$ is generated by $\langle \bar{\nu}_6+\epsilon_6\rangle$ and $i\circ ([\nu_5^3])$. Therefore, there are integers $a''$ and $b''$ such that 
\[
{21\cdot (\psi \circ s'')} \simeq {a''\cdot \langle \bar{\nu}_6+\epsilon_6 \rangle+b'' \cdot i\circ ([\nu_5^3])}. 
\]
Hence, we have 
\[
{21\cdot (p \circ \psi \circ s'')}\simeq  {a'' \cdot (p\circ  \langle \bar{\nu}_6+\epsilon_6\rangle)+b'' \cdot (p\circ i \circ  [\nu_5^3]) } \simeq a'' \cdot (\bar{\nu}_6+\epsilon_6). \qedhere 
\]
\end{proof}

To compute $p \circ \psi\circ s'' \circ \Sigma^3 (p_{11}\circ q \circ \langle \alpha\rangle)$ for $\alpha=\bar{\nu}_6+\epsilon_6$, we first compute $p \circ \psi\circ \Sigma^3 i_6 \circ s' \circ \Sigma^3 (q \circ \langle \alpha\rangle)$. For the sake of notational simplicity, let 
\[
s'(\alpha)=s' \circ \Sigma^3 (q \circ \langle \alpha\rangle).
\]
To handle $p \circ \psi\circ \Sigma^3 i_6 \circ s'(\alpha)$, we use the following lemma.


\begin{lemma}\label{lemma:3.4} Suppose that $\langle \alpha \rangle \in \pi_s(G_2)$.
We have
$\Sigma^3 (p_6 \circ i_6) \circ s'(\alpha)\simeq \Sigma^3 \alpha$.
\end{lemma}

\begin{proof}
Since $\Sigma^3 i_6 \circ s' + s'' \circ \Sigma^3 p_{11}$ is homotopic to the identity map, we have
\begin{align*}
\Sigma^3 p_6&\simeq \Sigma^3 p_6 \circ (\Sigma^3 i_6 \circ s' + s'' \circ \Sigma^3 p_{11})\\
&\simeq \Sigma^3 p_6 \circ \Sigma^3 i_6 \circ s' + \Sigma^3 p_6 \circ s'' \circ \Sigma^3 p_{11}.
\end{align*}
Since 
$\Sigma^3 p_6 \circ s''$ is in $\pi_{14}(S^9)\cong \{0\}$, we obtain  $\Sigma^3 p_6\simeq \Sigma^3 (p_6 \circ  i_6) \circ s'$. 
Therefore, we have \[
\Sigma^3 (p_6 \circ  i_6) \circ s'\circ \Sigma^3 (q \circ \langle \alpha \rangle)
\simeq \Sigma^3 (p_6 \circ q \circ \langle \alpha \rangle)
\simeq \Sigma^3 (p \circ \langle \alpha \rangle)
\simeq \Sigma^3 \alpha. \qedhere
\]
\end{proof}


Next, we factorize the map $p\circ \psi \circ \Sigma^3 i_6$ through $S^7 \vee S^9$ to prove the following proposition.

\begin{proposition}
\label{proposition:3.5}
There are integers $b'$ and $c'$ such that  for $\langle \alpha \rangle \in \pi_s(G_2)$, there is a homotopy class $g(\alpha)\in \pi_{s+3}(S^7)$ satisfying  the identity 
\begin{align*}
3 \cdot p\circ \psi \circ \Sigma^3 i_6 \circ s' (\alpha) &\simeq b' \cdot \eta_6\circ  g(\alpha)+c' \cdot \nu_6\circ \Sigma^3 \alpha.
\end{align*}
\end{proposition}

\begin{proof}
Let us consider the mapping cone $j(\eta_3)\colon S^3 \to C(\eta_3)$ of $\eta_3\colon S^4\to S^3$. 
We denote by $\pi(\eta_3)\colon C(\eta_3)\to S^5$ the pinch map collapsing the subspace $S^3\subset C(\eta_3)$ to the base point. 
Let 
$\beta_0\colon S^5 \to P^6(2)$ be the inclusion map of the bottom cell.
Let $\beta=\beta_0 \circ \pi(\eta_3)$, so that $i_6 \circ \beta \simeq i_5 \circ \pi(\eta_3)$.
Since $p\circ \psi \circ \Sigma^3  i_5$ is in $\pi_{8}(S^6)\cong\mathbb{Z}/2\{ \eta_6^2\}$, there is an integer $a'$ such that $p\circ \psi \circ \Sigma^3 i_5\simeq a' \cdot \eta_6^2$. Hence, we have
\[
p\circ \psi \circ \Sigma^3 i_5 \circ \Sigma^3 \pi(\eta_3)\simeq a' \cdot \eta_6^2 \circ \Sigma^3 \pi(\eta_3)\simeq a' \cdot \eta_6 \circ \Sigma \eta_6 \circ \pi(\eta_6)\simeq 0.
\]
It implies that  $p\circ \psi \circ \Sigma^3 i_6\circ \Sigma^3 \beta\simeq 0$. Therefore, there is a map ${\psi}' \colon \Sigma^3 C(\beta) \to S^6$ such that 
\begin{align*}
p\circ \psi \circ \Sigma^3 i_6&\simeq {\psi}' \circ \Sigma^3 j(\beta).
\end{align*}


Since the mod $2$ reduced cohomology group of $C(\beta)$ is isomorphic to $\mathbb{Z}/2\{ x_4, x_6\}$, where $x_4$, $x_6$ are generators of degree $4$ and $6$, respectively, and since $\mathrm{Sq}^2 x_4=0$, the finite complex $C(\beta)$ is homotopy equivalent to $S^4 \vee S^6$. Hence, there is a map ${p}'_4\colon C(\beta)\to S^4$ such that the induced homomorphism in integral homology $H_{4}(C(\beta);\mathbb{Z})\to H_{4}(S^4;\mathbb{Z})\cong \mathbb{Z}$ is an isomorphism. Since $C(\eta_3)$ is a $5$-dimensional finite complex, the composition ${p_6 \circ i_6 \circ \beta} \colon C(\eta_3)\to S^6$ is null homotopic. 
Therefore, there is a map ${p}'_6 \colon C(\beta)\to S^6$ such that the composition ${{p}'_6 \circ j(\beta)}$ is homotopic to the composition ${p_6\circ i_6}$.
There is an indeterminacy in the choice of ${p}'_6$.
The indeterminacy is given by the image of the induced map $[\Sigma C(\eta_3), S^6] \to [C(\beta), S^6]$.
The image is $2\cdot \mathbb{Z}\subset \mathbb{Z}\cong [C(\beta), S^6]$. Thus, by choosing a suitable null homotopy of ${p_6 \circ i_6\circ \beta}$, 
we may choose $p'_6$ that induces an isomorphism in integral homology groups $(p_6')_*\colon H_6(C(\beta);\mathbb{Z})\to H_6(S^6;\mathbb{Z})$.
Then, the map $\Sigma^3 {p}'_4 \;\dot{\vee}\; \Sigma^3 {p}'_6 \colon \Sigma^3 C(\beta) \to S^7\vee S^9$ induces an isomorphism in integral homology groups. Since these finite complexes are simply-connected, by the Whitehead theorem, $\Sigma^3 {p}'_4 \;\dot{\vee}\; \Sigma^3 {p}'_6$ is a homotopy equivalence. Let $\varphi' \colon S^7\vee S^9 \to \Sigma^3 C(\beta)$ be a map such that  the composition $
\varphi' \circ (\Sigma^3 {p}_4' \;\dot{\vee}\; \Sigma^3 {p}'_6)$ is homotopic to the identity map of $\Sigma^3 C(\beta)$. Then, we have 
\begin{align*}
p\circ \psi \circ \Sigma^3 i_6 
&\simeq {\psi}' \circ \varphi' \circ (\Sigma^3 {p}'_4 \;\dot{\vee}\; \Sigma^3 {p}'_6) \circ \Sigma^3 j(\beta).
\end{align*}


We proceed to describe ${p\circ \psi\circ \Sigma^3 i_6 \circ s'(\alpha)}$. Let us write $g(\alpha)$ for ${\Sigma^3 ({p}_4' \circ j(\beta))\circ s'(\alpha)}$. 
Since $s'$ is a suspension, so is $\Sigma^3 j(\beta) \circ s'(\alpha)$. Hence, we have 
\begin{align*}
p\circ \psi \circ \Sigma^3 i_6 \circ s'(\alpha)
&\simeq {\psi}' \circ \varphi' \circ \left( g(\alpha)  \;\dot{\vee}\; \Sigma^3 ({p}'_6 \circ j(\beta)) \circ   s'(\alpha)\right).
\end{align*}
Since ${p}'_6 \circ j(\beta)\simeq p_6 \circ i_6$, we have
\begin{align*}
p\circ \psi \circ \Sigma^3 i_6 \circ s'(\alpha)
&\simeq {\psi}' \circ \varphi' \circ \left(g(\alpha) \;\dot{\vee}\; \Sigma^3 \alpha\right)
\end{align*}
by Lemma~\ref{lemma:3.4}.
Since the homotopy class of ${\psi}' \circ \varphi'$ belongs to $\pi_{7}(S^6) \oplus \pi_{9}(S^6)\cong \mathbb{Z}/2\oplus \mathbb{Z}/24$,  the homotopy class of ${3\cdot {\psi}' \circ \varphi'}$ is in the $2$-primary component ${}_2(\pi_{7}(S^6) \oplus \pi_{9}(S^6)) \cong \mathbb{Z}/2\{ \eta_6\}\oplus \mathbb{Z}/8\{ \nu_6\}$.
Hence, there are integers $b'$, $c'$ such that ${ 3 \cdot {\psi}' \circ \varphi'} \simeq (b' \cdot \eta_6) \,\dot{+}\, (c' \cdot \nu_6)$.
Thus, we have
\begin{align*}
3 \cdot p\circ \psi \circ \Sigma^3 i_6 \circ s' (\alpha) &\simeq ((b' \cdot \eta_6)\,\dot{+}\, (c' \cdot \nu_6))  \circ \left( g(\alpha) \;\dot{\vee}\; \Sigma^3 \alpha  \right)
\\
&\simeq b' \cdot \eta_6\circ  g(\alpha)+c' \cdot \nu_6\circ \Sigma^3 \alpha. \qedhere
\end{align*}
\end{proof}

Using the identity in Proposition~\ref{proposition:3.5}, we prove the following proposition.


\begin{proposition}\label{proposition:3.6} For $\alpha=\bar{\nu}_6+\epsilon_6$, we have 
$p\circ \psi \circ \Sigma^3 i_6 \circ s'(\alpha)\simeq 0$.
\end{proposition}

\begin{proof}
Let $b'$ be the integer $b'$ in Proposition~\ref{proposition:3.5}.
For $\alpha=\bar{\nu}_6+\epsilon_6$, since $\alpha$ is in the $2$-primary component, $g(\alpha)$ is also in the $2$-primary component.  By \cite{toda-1962}*{Theorem 7.3}, we have $_{2}\pi_{17}(S^7)\cong \mathbb{Z}/8\{ \nu_7\circ \sigma_{10} \}\oplus \mathbb{Z}/2\{ \eta_7\circ \mu_8\}$ . Hence, there are integers $a_1, a_2$ such that 
$g(\alpha)\simeq a_1 \cdot \nu_7 \circ \sigma_{10}+a_2\cdot  \eta_7 \circ \mu_8$.
By Proposition~\ref{proposition:2.1}, we have $\eta_6\circ g(\alpha)\simeq 4 a_2\cdot  \zeta_6$ and
$\nu_6 \circ \Sigma^3 \alpha\simeq  \nu_6 \circ (\bar{\nu}_9+\epsilon_9)\simeq 4\cdot  \bar{\nu}_6\circ \nu_{14}\simeq 0$.
Hence, by Proposition~\ref{proposition:3.5}, we have
\[
{3\cdot p\circ \psi \circ \Sigma^3 i_6 \circ s'(\alpha)} \simeq 4 a_2 b' \cdot  \zeta_6.
\]
On the other hand, by Proposition~\ref{proposition:3.3}, the map
${21 \cdot p\circ \psi \circ s''}$ is homotopic to a scalar multiple of $ (\bar{\nu}_6+\epsilon_6)$ and 
$\Sigma^3 (p_{11}\circ q\circ \langle \alpha\rangle)$ is a scalar multiple of $\nu_{14}$. Hence, 
${21\cdot p\circ \psi \circ s''\circ \Sigma^3 (p_{11}\circ q \circ \langle \alpha\rangle)}$ is a scalar multiple of $(\bar{\nu}_6+\epsilon_6)\circ \nu_{14}\simeq \bar{\nu}_6 \circ \nu_{14}$. Therefore, 
we have 
\[
{21\cdot p\circ \psi \circ \Sigma^3 (q \circ \langle \alpha\rangle)}\equiv 28 a_2 b'\cdot \zeta_6\mod (\bar{\nu}_6\circ \nu_{14}).
\]
By Proposition~\ref{proposition:3.2}, it is congruent to $0$ modulo $(\bar{\nu}_6\circ \nu_{14})$. Hence, we have $28 a_2 b' \equiv 4a_2b' \equiv 0 \mod (8)$ and we obtain 
$
{3\cdot p\circ \psi \circ \Sigma^3 i_6 \circ s'(\alpha)} \simeq 0.
$
Since $\alpha=\bar{\nu}_6+\epsilon_6$ is in the $2$-primary component, $p\circ \psi \circ \Sigma^3 i_6 \circ s'(\alpha)$ is also in the $2$-primary component. Since the multiplication by $3$ on the $2$-primary component is injective, we obtain $p\circ \psi \circ \Sigma^3 i_6 \circ s'(\alpha)\simeq 0$.
\end{proof}

Now, we are ready to compute $p\circ \psi\circ s''\circ \Sigma^3(p_{11}\circ q\circ \langle \alpha\rangle)$ for $\alpha=\bar{\nu}_6+\epsilon_6$ using Propositions~\ref{proposition:3.2} and \ref{proposition:3.6}.


\begin{proposition}\label{proposition:3.7}
We have
$
p\circ \psi \circ s'' \circ \Sigma^3(p_{11}\circ q\circ \langle \bar{\nu}_6+\epsilon_6 \rangle)\simeq 2\bar{\nu}_6\circ \nu_{14}.
$
\end{proposition}

\begin{proof}
By Proposition~\ref{proposition:3.6}, for $\alpha=\bar{\nu}_6+\epsilon_6$, we have
\begin{align*}
p \circ \psi \circ \Sigma^3 (q\circ \langle \alpha \rangle)&\simeq p \circ \psi \circ \Sigma^3 i_6 \circ s'( \alpha)
+p \circ \psi \circ s''\circ \Sigma^3 (p_{11}\circ q\circ \langle\alpha\rangle)
\\
&\simeq p \circ \psi \circ s''\circ \Sigma^3 (p_{11}\circ q\circ \langle\alpha\rangle). \end{align*}
Thus, we have the desired result by Proposition~\ref{proposition:3.2}.
\end{proof}


Next, we use Proposition~\ref{proposition:3.7} to complete the proof of Proposition~\ref{proposition:1.5}. There is a cofiber sequence
\[
S^5 \stackrel{i_5}{\longrightarrow} V \stackrel{j(i_5)}{\longrightarrow}  C(i_5) \stackrel{\pi(i_5)}{\longrightarrow} S^6.
\]
We consider a factorization of $\pi(i_5)\circ j(i_5) \circ q\circ \langle \alpha \rangle\simeq 0$ through $S^6 \vee S^{11}$ under the assumption that $\langle \alpha \rangle \in \pi_s(G_2)$ for $s<16$ to obtain the identity in the following proposition.

\begin{proposition}\label{proposition:3.8}
There is an integer $a$ such that, if $\langle \alpha \rangle \in \pi_s(G_2)$ and $s<16$, then we have
\[
(2\iota_6)\circ \alpha + a\Delta(\iota_{13}) \circ p_{11} \circ q \circ \langle \alpha \rangle\simeq 0.
\]
\end{proposition}

\begin{proof}
It is clear that there are maps ${p}''_6\colon C(i_5) \to S^6$, and ${p}''_{11}\colon C(i_5) \to S^{11}$ such that ${{p}''_{6}\circ j(i_5)} \simeq p_6$ and ${{p}''_{11}\circ j(i_5)}\simeq p_{11}$. The inclusion map $i''\colon S^6 \vee S^{11}\to S^6\times S^{11}$ is a $16$-equivalence. Since $C(i_5)$ is $11$-dimensional finite complex, there is a map $\varphi'' \colon C(i_5)\to S^{6}\vee S^{11}$ such that $i''\circ \varphi'' \simeq ({p}''_6 \times {p}''_{11})\circ d_C$, where $d_C$ is the diagonal map of $C(i_5)$. On the one hand, we have
\[
i''\circ \varphi'' \circ j(i_5)\circ q \circ \langle \alpha \rangle \simeq ({p}''_6 \times {p}''_{11})\circ d_C \circ j(i_5) \circ q\circ \langle\alpha \rangle\simeq (p_6 \times p_{11}) \circ d_V\circ q \circ \langle \alpha \rangle,
\]
where $d_V$ is the diagonal map of $V$.
On the other hand, we have
\[
 i'' \circ (\alpha \,\dot{\vee}\, p_{11} \circ q\circ \langle \alpha \rangle)\simeq (p_6 \times p_{11})\circ d_V \circ q \circ \langle \alpha \rangle.
 \]
Since $i''$ is a $16$-equivalence, if $\langle \alpha \rangle\in \pi_s(G_2)$ and $s<16$, we have
\[
\varphi'' \circ j(i_5)\circ q \circ \langle \alpha \rangle \simeq \alpha \,\dot{\vee}\, (p_{11}\circ q \circ \langle \alpha \rangle).
\]
It is also clear that $\varphi''$ is a homotopy equivalence. Let  $\varphi''' \colon S^6 \vee S^{11}\to C(i_5)$ be the homotopy inverse of the map $\varphi''$. Since the restriction of  $\pi(i_5) \circ \varphi''' $ to $S^6$ is homotopic to $2\iota_6$ and since $\pi_{11}(S^6)$ is generated by $\Delta(\iota_{13})$,  there is an integer $a$ such that $\pi(i_5) \circ \varphi''' \simeq (2\iota_6) \dot{+} a \Delta(\iota_{13})$.
Thus, we have 
\[
\pi(i_5)\circ j(i_5)\circ q \circ \langle \alpha \rangle \simeq (2\iota_6)\circ \alpha + a\Delta(\iota_{13}) \circ p_{11} \circ q \circ \langle \alpha \rangle\simeq 0.\qedhere
\]
\end{proof}

With the identity in Proposition~\ref{proposition:3.8}, we have the following proposition.


\begin{proposition} \label{proposition:3.9} Let $a$ be as in Proposition~\ref{proposition:3.8}. The following holds{\rm :}\\
{\rm (1)} There is an integer $b$ such that the composition 
$p_{11}\circ q \circ \langle 2 \Delta(\iota_{13})\rangle$ is homotopic to $b \iota_{11}$ in $\pi_{11}(S^{11})\cong \mathbb{Z}$ and $a b=-8$.\\
{\rm (2)} If $a\in \{ \pm 1\}$,
then the composition 
$p_{11}\circ q\circ \langle \bar{\nu}_6+\epsilon_6\rangle$ is homotopic to $c \nu_{11}$ in ${}_2\pi_{14}(S^{11})\cong \mathbb{Z}/8\{ \nu_{11}\}$ where $c$ is an integer such that $2c\equiv 4 \mod (8)$.
\end{proposition}

\begin{proof}
(1) Let $b$ be the integer such that $p_{11}\circ q\circ \langle 2\Delta(\iota_{13})\rangle\simeq b \iota_{11}$. By Proposition~\ref{proposition:3.8}, we have
\[
(2\iota_6) \circ 2 \Delta(\iota_{13})+a\Delta(\iota_{13}) \circ b \iota_{11} \simeq 0.
\]
By Proposition~\ref{proposition:2.2}, we have $(2\iota_6) \circ 2 \Delta(\iota_{13})\simeq 8 \Delta(\iota_{13})$. Hence, the right-hand side is 
$(8+ab)\cdot  \Delta(\iota_{13})\in \pi_{11}(S^6)\cong  \mathbb{Z}\{ \Delta(\iota_{13})\}$. Therefore, we have $8=-ab$. Furthermore, since $a$ and $b$ are integers, we have $a\in \{ \pm 1, \pm 2, \pm 4, \pm 8\}$. 
\\
(2) Since $\langle \bar{\nu}_6+\epsilon_6\rangle$ belongs to the $2$-primary component, so is $p_{11}\circ q\circ \langle \bar{\nu}_6+\epsilon_6\rangle$. Let $c$ be the integer such that $p_{11}\circ q\circ \langle \bar{\nu}_6+\epsilon_6\rangle\simeq c \nu_{11}$. 
Then, by Proposition~\ref{proposition:3.8}, we have
\[
(2\iota_6) \circ (\bar{\nu}_6+\epsilon_6)+a\Delta(\iota_{13}) \circ c \nu_{11} \simeq 0.
\]
By Proposition~\ref{proposition:2.2}, we have $(2\iota_6)\circ \bar{\nu}_6\simeq 4 \bar{\nu}_6$. Since $\epsilon_6$ is a suspension, we have $(2\iota_6)\circ \epsilon_6 \simeq 2\epsilon_6\simeq 0$. Moreover, by Proposition~\ref{proposition:2.1}, we have $\Delta(\iota_{13})\circ \nu_{11}\simeq 2 \bar{\nu}_6$. Hence, 
the left-hand side is equal to
$(4+2ac)\cdot \bar{\nu}_6 \in {}_2\pi_{14}(S^6)\cong \mathbb{Z}/8\{ \bar{\nu}_6\} \oplus \mathbb{Z}/2\{\epsilon_6\}$. Since $a\in \{\pm 1\}$,  we have $4 \pm 2 c\equiv 0 \mod (8)$.
From this, we immediately obtain $2c\equiv \pm 4 \equiv 4 \mod (8)$.
\end{proof}


To prove Proposition~\ref{proposition:1.5}, we assume the following Proposition~\ref{proposition:3.10}. 
We prove it in Section~\ref{sec:5}.

\begin{proposition}\label{proposition:3.10}
For any $\alpha \in \pi_{11}(G_2)$, 
the image of the induced homomorphism 
\[
(p_{11}\circ q \circ \alpha )_*\colon H_{11}(S^{11};\mathbb{Z})\to H_{11}(S^{11};\mathbb{Z})
\]
is contained in $8 \cdot H_{11}(S^{11};\mathbb{Z})$.
\end{proposition}


\begin{proof}[Proof of Proposition~\ref{proposition:1.5}]
Let $a$, $b$, $c$ be as in Propositions~\ref{proposition:3.8} and \ref{proposition:3.9}. Let $a''$ be as in Proposition~\ref{proposition:3.3}. 
By Proposition~\ref{proposition:3.9} {\rm (1)}, $a$ is an element in $\{ \pm 1, \pm 2, \pm 4, \pm 8\}$. By Proposition~\ref{proposition:3.10}, $b$ is divisible by $8$. Together with $ab=-8$, this implies $b\in \{\pm 8\}$. Hence, we have $a\in \{\pm 1\}$.
By Propositions~\ref{proposition:3.3} and \ref{proposition:3.9} (2), we have
\[
21 \cdot p\circ \psi \circ s'' \circ \Sigma^3 (p_{11}\circ q\circ \langle \bar{\nu}_6+\epsilon_6 \rangle ) \simeq a'' \cdot (\bar{\nu}_6+\epsilon_6)\circ c\nu_{14}
\] 
where $2c\equiv 4 \mod (8)$.  Since $2c\equiv 4 \mod (8)$, there is an odd integer $c_1$ such that $c=2c_1$. 
By Proposition~\ref{proposition:2.1}, the right-hand side is homotopic to $a'' c\cdot \bar{\nu}_6\circ \nu_{14}$. By Proposition~\ref{proposition:3.7}, the left-hand side is homotopic to $2\bar{\nu}_6 \circ \nu_{14}$. Since the order of $\bar{\nu}_6 \circ \nu_{14}$ is $4$, we have $2a'' c_1\equiv 2\mod (4)$ . It implies 
that $a''$ is an odd integer. 
By Proposition~\ref{proposition:2.5}, we obtain Proposition~\ref{proposition:1.5}. 
\end{proof}


\section{Proof of Proposition~\ref{proposition:3.1}}\label{sec:4}

In this section, we prove Proposition~\ref{proposition:3.1}. We use relative Samelson products rather than ordinary Samelson products. For the details of the relative Samelson product, we refer the reader to James' book \cite{james-book}*{Section 15}. 

Let $G$ be a compact Lie group and $H$ its closed subgroup. We consider the following fiber sequence
\[
H \stackrel{i}{\longrightarrow} G \stackrel{p}{\longrightarrow} G/H.
\]
Then, there is the relative Samelson product 
$
\langle \ , \ \rangle_r \colon \pi_s(H) \times \pi_t(G/H)\to \pi_{s+t}(G/H).
$
The following proposition gives the relation between the Samelson product and the relative Samelson product.


\begin{proposition}\label{proposition:4.1}
For $\alpha\in \pi_s(H)$, $\beta \in \pi_t(G)$, we have
$
p\circ \langle i\circ \alpha, \beta\rangle\simeq \langle  \alpha, p\circ \beta\rangle_r
$.
\end{proposition}

The following proposition \cite{james-book}*{(16.5)} is what we need in the computation of relative Samelson products.


\begin{proposition}\label{proposition:4.2}
For $\alpha\in \pi_s(H)$, $\beta \in \pi_t(G/H)$, $\gamma\in \pi_{u}(S^t)$, we have 
\begin{align*}
\langle  \alpha , \beta \circ \gamma \rangle_r&\simeq \langle  \alpha, \beta \rangle_r \circ \Sigma^s \gamma + [  \langle  \alpha, \beta\rangle_r, \beta]\circ \Sigma^s \gamma'+ [[  \langle  \alpha, \beta \rangle_r, \beta], \beta ]\circ \Sigma^s \gamma''+\cdots
\end{align*}
where $\gamma'\simeq H(\gamma)$, $\gamma'', \dots $ are  the generalized Hopf invariants of $\gamma$ and $[ -, -]$ is the Whitehead product. 
\end{proposition}

In the rest of this section, we consider the case $G=G_2$, $H=SU(3)$ and $G/H=S^6$ for the relative Samelson product $\langle -, - \rangle_r$.


\begin{proposition} \label{proposition:4.3} We have
$\langle i_3, \bar{\nu}_6  \rangle_r\simeq 0$ and $\langle i_3, \epsilon_6  \rangle_r\simeq 2\bar{\nu}_6\circ \nu_{14}$.
\end{proposition}

\begin{proof}
Put $\alpha=i_3$, $\beta=\iota_6$ in Proposition~\ref{proposition:4.2}. Suppose that $\gamma \in \pi_{14}(S^6)$.
Then, we have $\langle i_3, \gamma\rangle_r\in \pi_{17}(S^6)$ and 
\[
\langle i_3, \iota_6\rangle_r\in \pi_9(S^6), 
[\langle i_3, \iota_6 \rangle_r, \iota_6] \in \pi_{14}(S^6), [[\langle i_3, \iota_6\rangle_r, \iota_6],\iota_6] \in \pi_{19}(S^6), \dots.
\]
For degree reasons, we have $\gamma''\simeq \gamma'''\simeq \cdots\simeq 0$ and
\begin{align*}
\langle  i_3 , \gamma \rangle_r\simeq \langle  i_3 , \iota_6 \circ \gamma \rangle_r&\simeq \langle  i_3, \iota_6\rangle_r \circ \Sigma^3 \gamma + [  \langle  i_3, \iota_6\rangle_r, \iota_6]\circ \Sigma^3 H(\gamma).
\end{align*}
There is an integer $a$ such that $\langle  i_3 , \iota_6 \rangle_r\simeq (1+2a)\cdot \nu_6$ by \cite{oshima-2005}*{(4.7)}.
Furthermore, we have $[\nu_6, \iota_6]\simeq \pm [\iota_6, \iota_6]\circ \nu_{11}\simeq \pm \Delta(\iota_{13})\circ \nu_{11}\simeq \pm 2\bar{\nu}_6$.
Hence, we have
\begin{align*}
\langle  i_3 , \gamma \rangle_r&\simeq (1+2a) \cdot \nu_6  \circ \Sigma^3 \gamma \pm 2(1+2a) \cdot \bar{\nu}_6 \circ \Sigma^3 H(\gamma).
\end{align*}
For $\gamma=\bar{\nu}_6$, $\Sigma^3 \gamma \simeq \bar{\nu}_9$ and $H(\gamma)\simeq (1+2b) \cdot \nu_{11}$ for some integer $b$. Thus, we have
\begin{align*}
\langle  i_3 , \bar{\nu}_6  \rangle_r&\simeq (1+2a)\cdot \nu_6\circ \bar{\nu}_9\pm 2 (1+2a)\cdot  \bar{\nu}_6 \circ (1+2b) \nu_{14}
\\
& \simeq 2(1+2a) \cdot \bar{\nu}_{6}\circ \nu_{14}\pm 2(1+2a)(1+2b)\cdot \bar{\nu}_6 \circ \nu_{14}\simeq 0.
\end{align*}
For $\gamma=\epsilon_6$, $\Sigma^3 \gamma \simeq \epsilon_9$ and $H(\gamma)\simeq 0$
since $\epsilon_6$ is a suspension. Thus, we have
\[
\langle  i_3 , \epsilon_6  \rangle_r\simeq (1+2a)\cdot \nu_6\circ \epsilon_9
\simeq 2(1+2a) \cdot \bar{\nu}_{6}\circ \nu_{14}\simeq 2 \bar{\nu}_{6}\circ \nu_{14}.\qedhere
\]
\end{proof}

Now, we complete the proof of Proposition~\ref{proposition:3.1}. 

\begin{proof}[Proof of Proposition~\ref{proposition:3.1}]
By Propositions~\ref{proposition:4.1} and \ref{proposition:4.3}, we have  
$
p\circ \langle i_3, \langle \bar{\nu}_6+\epsilon_6 \rangle\rangle  \simeq
\langle i_3, \bar{\nu}_6+\epsilon_6  \rangle_r\simeq
\langle i_3, \bar{\nu}_6  \rangle_r+
\langle i_3, \epsilon_6  \rangle_r
 \simeq 2\bar{\nu}_6\circ \nu_{14}.
$
\end{proof}


\section{Proof of Proposition~\ref{proposition:3.10}}
\label{sec:5}

In this section, to prove Proposition~\ref{proposition:3.10}, we consider the following factorization of $\alpha \colon S^{11}\to G_2$ by $n$-connective covers $G_2\langle n \rangle$ of $G_2$:
\[
S^{11}\stackrel{\beta}{\longrightarrow} G_2\langle 9\rangle \stackrel{g_3}{\longrightarrow} G_2\langle 8\rangle
\stackrel{g_2}{\longrightarrow}  G_2\langle 3\rangle \stackrel{g_1}{\longrightarrow}  G_2.
\]
The fibers of the maps $g_1$, $g_2$, and $g_3$ are $2$-locally homotopy equivalence to $K(\mathbb{Z},2)$, $K(\mathbb{Z}/2,7)$, and $K(\mathbb{Z}/2, 8)$, respectively.
The rational cohomology groups of $G_2\langle 3\rangle$, $G_2\langle 8\rangle$ and $G_2\langle 9\rangle$ are isomorphic to the rational cohomology group of $S^{11}$. Moreover, the maps $\beta$, $g_2$ and $g_3$ induce isomorphisms in rational cohomology groups.
It is easy to see that the torsion-free parts of the $11$-th integral cohomology groups of the above connective covers are isomorphic to $\mathbb{Z}$. We denote the mod $2$ reduction of generators of these groups by $z_1$ for $G_2$, $z_2$ for $G_2\langle 3 \rangle$ and $z_3$ for $G_2\langle 8\rangle$, respectively. These mod $2$ cohomology classes are unique. We compute $g_1^{*}(z_1)$, $g_2^{*}(z_2)$, $g_3^{*}(z_3)$ to prove Proposition~\ref{proposition:3.10}.

The mod $2$ cohomology ring of $G_2$ is isomorphic to $\mathbb{Z}/2[ x_3,x_5]/(x_3^4, x_5^2)$. By computing the mod $2$ Bockstein spectral sequence, it is easy to see that $z_1=x_3^2 x_5$.
 As an algebra over the mod $2$ Steenrod algebra $\mathcal{A}$, it is generated by the cohomology class $x_3$. Since the induced homomorphism $g_1^*\colon \widetilde{H}^*(G_2\langle 3 \rangle;\mathbb{Z}/2)\to \widetilde{H}^{*}(G_2;\mathbb{Z}/2)$ maps $x_3$ to zero, the induced homomorphism $g_1^*$ itself is zero. Therefore, we have $g_1^*(z_1)=0$. It implies the image of the induced homomorphism in integral homology groups $H_{11}(G_2\langle 3 \rangle;\mathbb{Z})\to H_{11}(G_2;\mathbb{Z})\cong \mathbb{Z}$ is contained in $2\cdot H_{11}(G_2;\mathbb{Z})$.

For the $3$-connective cover $G_2\langle 3 \rangle$, we have the following proposition.


\begin{proposition}
[\cite{mimura-1967}*{Theorem 2.3}]\label{proposition:5.1}
For the $3$-connective cover $G_2\langle 3 \rangle$ of the Lie group $G_2$, we have
$
H^{*}(G_2\langle 3\rangle;\mathbb{Z}/2) \cong \mathbb{Z}/2[ y_8] \otimes \Lambda(\mathrm{Sq}^1 y_8, \mathrm{Sq}^2\mathrm{Sq}^1 y_8)
$.
\end{proposition}

By Proposition~\ref{proposition:5.1},
up to degree $13$, the mod $2$ reduced cohomology group of $G_2\langle 3\rangle$ is isomorphic to $\mathbb{Z}/2\{ y_8, \mathrm{Sq}^1 y_8, \mathrm{Sq}^2 \mathrm{Sq}^1 y_8\}$. It is also easy to see that $H^{11}(G_2\langle 3 \rangle;\mathbb{Z})\cong \mathbb{Z}$ and  $z_2=\mathrm{Sq}^2 \mathrm{Sq}^1 y_8$.

For the $8$-connective cover $G_2\langle 8 \rangle$, we prove the following proposition.


\begin{proposition}\label{proposition:5.2}
We have $g_2^*(z_2)=0$. 
The mod $2$ reduced cohomology groups of $G_2\langle 8\rangle$ are isomorphic, up to degree $\leq 12$, to 
\[
\mathbb{Z}/2\{ y_9, \mathrm{Sq}^1 y_9, \mathrm{Sq}^2 y_9, \mathrm{Sq}^2\mathrm{Sq}^1  y_9,  y_{11}, \mathrm{Sq}^1 y_{11}\}.
\] Moreover, the integral cohomology group $H^{11}(G_2\langle 8\rangle;\mathbb{Z})$ is isomorphic to $\mathbb{Z}$ and $z_3=\mathrm{Sq}^2 y_9$. 
\end{proposition}

\begin{proof}
Since $g_2^*(y_8)=0$ and $g_2^{*}$ is an $\mathcal{A}$-module homomorphism, we have $g_2^*(z_2)=g_2^{*}(\mathrm{Sq}^2 \mathrm{Sq}^1 y_8)=\mathrm{Sq}^2 \mathrm{Sq}^1 g_2^{*}(y_8)=0$. 
Let us calculate the Leray-Serre spectral sequence associated with the fibration
\[
G_2\langle 8\rangle \stackrel{g_2}{\longrightarrow} G_2\langle 3\rangle,
\]
whose fiber is $2$-locally homotopy equivalent to $K(\mathbb{Z}/2, 7)$.
Let us denote by $u_7$ the generator of $H^{7}(K(\mathbb{Z}/2, 7);\mathbb{Z}/2)$. 
The $E_2$-page is given by 
\[
H^*(G_2\langle 3 \rangle;\mathbb{Z}/2)\otimes H^{*}(K(\mathbb{Z}/2, 7);\mathbb{Z}/2).
\] 
Elements $y_8\otimes 1$, $\mathrm{Sq}^1 y_8 \otimes 1$, and $\mathrm{Sq}^2 \mathrm{Sq}^1 y_8\otimes 1$ in the $E_2$-page of the spectral sequence are hit by some nontrivial differentials. The mod $2$ reduced cohomology group $\widetilde{H}^{*}(K(\mathbb{Z}/2, 7);\mathbb{Z}/2)$ is isomorphic to a free $\mathcal{A}$-module generated by $u_7$ up to degree $\leq 12$. For dimensional reasons, up to degree $\leq 9$, the only possible nontrivial differentials are $d_8(1\otimes u_7)=y_8\otimes1$, $d_9(1\otimes \mathrm{Sq}^1 u_7)=\mathrm{Sq}^1 y_8 \otimes 1$. Furthermore, for dimensional reasons, we have $d_r(1\otimes \mathrm{Sq}^2 u_7)=0$ for $r\geq 2$.
Since the differential $d_r\colon E_r^{s,*}\to E_r^{s+r, *}$ commutes with Steenrod squares, 
we have $d_r(1\otimes \mathrm{Sq}^3 u_7)=d_r(\mathrm{Sq}^1 (1\otimes \mathrm{Sq}^2 u_7))=\mathrm{Sq}^1 d_r(1\otimes \mathrm{Sq}^2 u_7)=0$. Therefore, we have $d_{11}(1 \otimes \mathrm{Sq}^2 \mathrm{Sq}^1 u_7)=\mathrm{Sq}^2 \mathrm{Sq}^1 y_8 \otimes 1$. These three nontrivial differentials are the only nontrivial differentials up to degree $\leq 12$.
Thus, up to degree $\leq 12$, both
$E_\infty^{0,*}$ and $H^{*}(G_2\langle 8 \rangle ; \mathbb{Z}/2)$ are isomorphic to the kernel of the evaluation map $\mathcal{A}\{ u_7\} \to 
\widetilde{H}^{*}(G_2\langle 3 \rangle;\mathbb{Z}/2)$ sending $u_7$ to $y_8$.
Up to degree $\leq 5$, 
$\mathcal{A}$ is spanned by
\[
1, \mathrm{Sq}^1, \mathrm{Sq}^2, \mathrm{Sq}^3, 
\mathrm{Sq}^2 \mathrm{Sq}^1, \mathrm{Sq}^4, \mathrm{Sq}^3 \mathrm{Sq}^1, \mathrm{Sq}^5, \mathrm{Sq}^4 \mathrm{Sq}^1.
\]
From Adem relations, we have $
\mathrm{Sq}^1 \mathrm{Sq}^2=\mathrm{Sq}^3$, 
$\mathrm{Sq}^2 \mathrm{Sq}^2=\mathrm{Sq}^3 \mathrm{Sq}^1$, 
$\mathrm{Sq}^3 \mathrm{Sq}^2 =0$, $\mathrm{Sq}^2\mathrm{Sq}^1 \mathrm{Sq}^2=\mathrm{Sq}^5+\mathrm{Sq}^4 \mathrm{Sq}^1$, 
and $\mathrm{Sq}^1 \mathrm{Sq}^4=\mathrm{Sq}^5$. We put $y_9=\mathrm{Sq}^2 u_7$, $y_{11}=\mathrm{Sq}^4 u_7$. Then, up to degree $\leq 12$, 
$E_\infty^{0,*}\cong H^{*}(G_2\langle 8\rangle;\mathbb{Z}/2)$ is isomorphic to
\[
\mathbb{Z}/2\{ y_9, \mathrm{Sq}^1 y_9, \mathrm{Sq}^2 y_9, \mathrm{Sq}^2\mathrm{Sq}^1  y_9,   y_{11}, \mathrm{Sq}^1 y_{11}\}.
\]
Since $\mathrm{Sq}^1 y_{11}\not=0$ and $\mathrm{Sq}^1 \mathrm{Sq}^2 y_9=0$, the integral cohomology group $H^{11}(G_2\langle 8 \rangle;\mathbb{Z})$ is isomorphic to $\mathbb{Z}$ and $z_3=\mathrm{Sq}^2 y_9$.
\end{proof}

We do not need to compute the mod $2$ cohomology groups of the connective cover $G_2\langle 9 \rangle$. The following proposition is all we need on the mod $2$ cohomology groups of $G_2\langle 9 \rangle$.


\begin{proposition}\label{proposition:5.3}
We have  $g_3^{*}(z_3)=0$.
\end{proposition}

\begin{proof}
Since $g_3^{*}(y_9)=0$ and $g_3^*$ is an $\mathcal{A}$-module homomorphism, we have $g_3^{*}(z_3)=g_3^{*}(\mathrm{Sq}^2 y_9)=\mathrm{Sq}^2 (g_3^*(y_9))=0$.
\end{proof}


\begin{proof}[Proof of Proposition~\ref{proposition:3.10}]
Since $g_3^{*}(z_3)=0$,  $g_2^{*}(z_2)=0$ and $g_1^*(z_1)=0$, we have 
\begin{align*}
{g_3}_*(H_{11}(G_2\langle 9\rangle;\mathbb{Z})/{\mathrm{Torsion}}) &\subset 2 \cdot H_{11}(G_2\langle 8\rangle;\mathbb{Z})/\mathrm{Torsion}, \\
{g_2}_*(H_{11}(G_2\langle 8\rangle;\mathbb{Z})/\mathrm{Torsion}) & \subset 2 \cdot H_{11}(G_2\langle 3\rangle;\mathbb{Z})/\mathrm{Torsion}, \\
{g_1}_*(H_{11}(G_2\langle 3\rangle;\mathbb{Z})/\mathrm{Torsion}) &\subset 2 \cdot H_{11}(G_2;\mathbb{Z})/\mathrm{Torsion}, 
\end{align*}
respectively. Since $H_{11}(G_2;\mathbb{Z})$ has no torsion, we have
\[
(g_1\circ g_2 \circ g_3)_*(H_{11}(G_2\langle 9\rangle;\mathbb{Z})) \subset 8 \cdot H_{11}(G_2;\mathbb{Z}).
\]
Furthermore, since
\[
p_{11}\circ q \circ \alpha \simeq p_{11}\circ q \circ g_1\circ g_2\circ g_3 \circ \beta, 
\]
we have 
\begin{align*}
{(p_{11}\circ q \circ \alpha )}_{*}(H_{11}(S^{11};\mathbb{Z}))& \subset {(p_{11}\circ q)}_*\circ (g_1\circ g_2\circ g_3)_{*} ({\beta}_{*}(H_{11}(S^{11};\mathbb{Z})))
\\
&\subset {(p_{11}\circ q)}_*(8\cdot H_{11}(G_2;\mathbb{Z}))
\\
& \subset 8 \cdot H_{11}(S^{11};\mathbb{Z}).
\end{align*}
It completes the proof of Proposition~\ref{proposition:3.10}.
\end{proof}


\section{Proof of Proposition~\ref{proposition:1.2}}\label{sec:6}

In this section, to prove Proposition~\ref{proposition:1.2}, we consider the fiber sequence
\[
\mathcal{G}_k \stackrel{h_k}{\longrightarrow} G_2 \stackrel{\partial_k}{\longrightarrow} \Omega_0^3 G_2.
\]
For a space $X$, we identify the homotopy set $[\Sigma^3 X, G_2]$ with $[X, \Omega_0^3 G_2]$ by taking the triple adjoints. Thus, we regard ${\partial_k}_{*}$ as a homomorphism from $[X, G_2]$ to $[\Sigma^3 X, G_2]$.
For an integer $a$ and a map $f\colon X\to G_2$, we define a map $a\cdot f\colon X\to G_2$ by $(a\cdot f)(x)=f(x)^a$ using the group structure of $G_2$. Similarly, for maps $f, g \colon X\to G_2$, we define a map $f+g$ by $(f+g)(x)=f(x) \cdot g(x)$ where $\cdot$ is the multiplication of $G_2$. As long as $X$ is a suspension of a space, for instance, in $[\Sigma^3 X, G_2]$, these operations are compatible with the corresponding operation used in previous sections.

First, we compute the $13$-th homotopy groups of the gauge group $\mathcal{G}_k$ to show that $\mathcal{G}_0$, $\mathcal{G}_1$, $\mathcal{G}_2$ have different $2$-local homotopy types.


\begin{proposition}\label{proposition:6.1}
We have $\pi_{13}(\mathcal{G}_k)\cong \mathbb{Z}/(8,2k) \oplus \mathbb{Z}/2$.
\end{proposition}

\begin{proof}
Since $\pi_{14}(BG_2)\cong \pi_{13}(G_2)\cong \{0\}$, we have the following exact sequence:
\[
\pi_{14}(G_2) \stackrel{{\partial_k}_{*}}{\longrightarrow} \pi_{17}(G_2) 
\longrightarrow \pi_{14}(B\mathcal{G}_k) \to \{0\}.
\]
By Propositions~\ref{proposition:2.1} and \ref{proposition:2.3}, we have
$\pi_{17}(G_2)={}_2\pi_{17}(G_2)\cong \mathbb{Z}/8\{ \langle \bar{\nu}_6\circ \nu_{14}\rangle \} \oplus \mathbb{Z}/2\{ \langle 4 \zeta_6\rangle\}$ and  ${}_2 \pi_{17}(S^6)\cong \mathbb{Z}/4\{ \bar{\nu}_6\circ \nu_{14}\} \oplus \mathbb{Z}/8\{ \zeta_6\}$.
Furthermore, the kernel of the induced homomorphism
$p_*\colon \pi_{17}(G_2)\to \pi_{17}(S^6)$ is generated by $4\cdot \langle \bar{\nu}_6\circ \nu_{14} \rangle$, and the image of ${\partial_1}_*\colon\pi_{14}(G_2)\to [\Sigma^3 G_{2,11}, G_2]$ is generated by ${\partial_1}_*(\langle \bar{\nu}_6+\epsilon_6\rangle)$ and ${\partial_1}_*(i_*([\nu_5^3]))$.
On the one hand, by Proposition~\ref{proposition:3.1}, we have 
\begin{align*}
p_*\circ  {\partial_1}_*(\langle\bar{\nu}_6+\epsilon_6\rangle)&\simeq p_* \langle i_3, \langle \bar{\nu}_6+\epsilon_6\rangle \rangle 
\simeq 2 \bar{\nu}_6\circ \nu_{14}.
\end{align*}
On the other hand, by the naturality of the Samelson product, we have 
\begin{align*}
p_*\circ {\partial_1}_*(i_*([\nu_5^3]))
& \simeq p_* (\langle i_3, i_*([\nu_5^3]) \rangle)
\simeq p_*\circ i_* (\langle i'_3, [\nu_5^3]\rangle )
\simeq 0,
\end{align*}
where $i_3'\colon S^3 \to SU(3)$ is the map such that $i\circ i_3'=i_3$.
Hence, ${\partial_1}_*(\pi_{14}(G_2))$ is generated by $2 \cdot \langle \bar{\nu}_6\circ \nu_{14}\rangle$. 
Since ${\partial_k}_*=k \cdot {\partial_1}_*$, the image of ${\partial_k}_*\colon \pi_{14}(G_2)\to \pi_{17}(G_2)$ is generated by $2k \cdot \langle \bar{\nu}_6\circ \nu_{14}\rangle$.
So, we have 
$\pi_{13}(\mathcal{G}_k)\cong \pi_{14}(B\mathcal{G}_k)\cong \mathbb{Z}/(8,2k) \oplus \mathbb{Z}/2$.
\end{proof}

To show $\mathcal{G}_4$ is not $2$-locally homotopy equivalent to $\mathcal{G}_0\simeq G_2\times \Omega_0^4 G_2$, we use the following Propositions~\ref{proposition:6.2}, \ref{proposition:6.3}, \ref{proposition:6.4} and \ref{proposition:6.5}.


\begin{proposition}\label{proposition:6.2}
The induced homomorphism ${h_k}_*\colon \pi_{3}(\mathcal{G}_k)_{(2)} \to \pi_{3}(G_2)_{(2)}$ is an isomorphism.
\end{proposition}

\begin{proof}
It follows from 
$\pi_3(\Omega_0^4 G_2)\cong \pi_7(G_2) \cong \{0\}$, $\pi_3(\Omega_0^3 G_2)_{(2)}\cong \pi_6(G_2)_{(2)} \cong \{0\}$, and the following long exact sequence associated with the fiber sequence:
\[
\pi_3(\Omega_0^4 G_2)\to \pi_3(\mathcal{G}_k) \to \pi_3(G_2) \to \pi_3(\Omega_0^3 G_2).\qedhere
\]
\end{proof}

The cell decomposition of $G_2$,
\[
  G_2 = S^3 \cup e^5 \cup e^6 \cup e^8 \cup e^9 \cup e^{11} \cup e^{14},
\]
comes from the fiber bundle $S^3 \to G_2 \to V$ or $SU(3) \to G_2 \to S^6$ with the usual cell decomposition of $V$ or $SU(3)$.
Let $G_{2,11}$ be the $11$-skeleton of $G_2$, and $i'\colon G_{2,11}\to G_{2}$ the inclusion map.
Let $q' =q\circ i'$. Let $p'\colon G_{2,11}\to G_{2,11}/G_{2,3}$ be the pinch map collapsing the subspace $G_{2,3}$ to the base point, where we write $G_{2,3}\simeq S^3$ for the $3$-skeleton of $G_2$.


\begin{proposition}\label{proposition:6.3}
For a map $f\colon G_{2,11}\to G_{2}$, there exist an integer $a$ and a map $g \colon G_{2,11}/G_{2,3}\to G_2$ such that 
$f\simeq (a \cdot i' )+(g\circ p')$.
\end{proposition}

\begin{proof}
The group 
$[G_{2,3},G_2]\cong \pi_3(G_2)\cong \mathbb{Z}$ is generated by the inclusion map $i_3\colon G_{2,3}\to G_{2}$.
Therefore, the induced homomorphism $[G_{2,11}, G_2]\to [G_{2,3}, G_2]$ is surjective. We obtain the result from the following exact sequence.
\[
\{0\} \leftarrow [G_{2,3}, G_2]\leftarrow [G_{2,11}, G_2]\leftarrow [G_{2,11}/G_{2,3}, G_2]. \qedhere
\]
\end{proof}

The map $p_{11}\circ q'$ is homotopic to the pinch map $G_{2,11}\to S^{11}$ collapsing the $10$-skeleton to the base point.
For $f \in [G_{2,11}, G_2]$, we have $3\cdot {\partial_k}_*(f)\simeq k\cdot  \psi \circ \Sigma^3 (q\circ f)$.
Since the identity map of $P^9(2)$ has order $4$, we have
\begin{align*}
21\cdot 3\cdot {\partial_4}_*(f)&\simeq 84\cdot \psi \circ s'' \cdot \Sigma^3 (p_{11}\circ q\circ f )+ 84\cdot \psi \cdot \Sigma^3 i_6 \circ s' \circ \Sigma^3 (q\circ f)\\
& \simeq 84\cdot \psi \circ s'' \circ \Sigma^3 (p_{11}\circ q\circ f).
\end{align*}


\begin{proposition}\label{proposition:6.4}
For the inclusion map $i'\colon G_{2,11}\to G_2$, the homotopy class 
$21 \cdot 3 \cdot {\partial_4}_*(i')$ has order $2$.
\end{proposition}

\begin{proof} 
Substituting $i'$ for $f$ in the above identity, we have
\[
21\cdot 3\cdot {{\partial_4}_*(i')}\simeq 84\cdot \psi\circ s''\circ {\Sigma^3 (p_{11}\circ q')}.
\]
In Section 3, Proposition~\ref{proposition:3.3}, we proved that there is an odd integer $a''$ such that $21\cdot p\circ \psi \circ s'' \simeq a'' (\bar{\nu}_6+\epsilon_6)$. The argument in the proof of Proposition~\ref{proposition:2.5} shows that $21\cdot 3\cdot {\partial_4}_*(i')$ has order $2$.
\end{proof}


\begin{proposition}\label{proposition:6.5}
For a map
$g\colon G_{2,11}/G_{2,3}\to G_2$, we have $21\cdot 3\cdot {\partial_4}_*(g\circ p')\simeq 0$.
\end{proposition}

\begin{proof}
Let $G_{2,10}$ be the $10$-skeleton of $G_2$. 
Let us consider the cofiber sequence
\[
G_{2,10} \to G_{2,11} \to S^{11}.
\]
It gives us an exact sequence of sets
\[
[G_{2,10}, S^{11}]\leftarrow [G_{2,11}, S^{11}]\leftarrow [S^{11}, S^{11}]
\]
Since $[G_{2,10}, S^{11}]$ is a single point, the induced map $[S^{11}, S^{11}]\to [G_{2,11}, S^{11}]$ is surjective.
It is also clear that 
the composition 
\[
[S^{11}, S^{11}]\to [G_{2,11}, S^{11}]\to \mathrm{Hom}\, (H_{11}(G_{2,11};\mathbb{Z}), H_{11}(S^{11};\mathbb{Z}))
\]
is bijective. The map $p_{11}\circ q'$ induces an isomorphism $(p_{11}\circ q')_*\colon H_{11}(G_{2,11};\mathbb{Z})\to H_{11}(S^{11};\mathbb{Z})$. Therefore, the homotopy set $[G_{2,11}, S^{11}]$ is the set consisting of $(a\cdot \iota_{11}) \circ p_{11} \circ q'$, where $a$ ranges over the set of integers. So, for the map $p_{11}\circ q\circ g \circ p'$, there is an integer $a$ such that $p_{11}\circ q\circ  g \circ p'\simeq (a\cdot \iota_{11}) \circ p_{11}\circ q'$. Since the map $g\circ p'$ factors through the $3$-connected space $G_{2,11}/G_{2,3}$, 
the map ${p_{11}\circ q \circ g \circ p'}$ factors through the $3$-connective cover $g_1\colon G_2\langle 3 \rangle \to G_2$. As we showed in Section~\ref{sec:5}, the induced homomorphism $g_1^*$ in mod $2$ reduced cohomology groups is zero and so we have $(g\circ p')_*(H_{11}(G_{2,11};\mathbb{Z}))\subset 2\cdot H_{11}(G_2;\mathbb{Z})$. Hence, there is an integer $b$ such that $a=2b$. Therefore, we have
\[
21\cdot 3\cdot {\partial_4}_*(g\circ p')\simeq b \cdot 168 \cdot \psi \circ s'' \circ \Sigma^3 (p_{11}\circ q')\simeq 0,
\]  since $\psi\circ s''\in \pi_{14}(G_2)\cong \mathbb{Z}/168\oplus \mathbb{Z}/2$. 
\end{proof}


Finally, we complete the proof of Proposition~\ref{proposition:1.2} using the above propositions.

\begin{proof}[Proof of Proposition~\ref{proposition:1.2}]
(2), (3), (4) and the ``if'' part of (1) are proved in {\ktt} \cite{ktt-2017}. It remains to show the ``only if'' part of (1). This is equivalent to saying that gauge groups $\mathcal{G}_0$, $\mathcal{G}_1$, $\mathcal{G}_2$ and $\mathcal{G}_4$ have distinct $2$-local homotopy types. By Proposition~\ref{proposition:6.1}, $\pi_{13}(\mathcal{G}_0)\cong \pi_{13}(\mathcal{G}_4)\cong \mathbb{Z}/8\oplus \mathbb{Z}/2$, $\pi_{13}(\mathcal{G}_1)\cong \mathbb{Z}/2 \oplus \mathbb{Z}/2$, $\pi_{13}(\mathcal{G}_2)\cong \mathbb{Z}/4 \oplus \mathbb{Z}/2$. Therefore, to complete the proof of the ``only if'' part of (1), we need to show that $\mathcal{G}_4\not \simeq_{(2)} G_2\times \Omega_0^4 G_2\simeq \mathcal{G}_0$. We prove it by contradiction.

Suppose that $\mathcal{G}_4$ is $2$-locally homotopy equivalent to $G_2\times \Omega_0^4G_2$. Then, there is a map $\phi\colon G_{2,11} \to \mathcal{G}_4$ such that the induced homomorphism $\phi_*\colon \pi_3(G_{2,11})_{(2)}\to \pi_3(\mathcal{G}_4)_{(2)}$ is an isomorphism. By Proposition~\ref{proposition:6.2}, the induced homomorphism $(h_4\circ \phi)_*\colon \pi_{3}(G_{2,11})_{(2)} \to \pi_{3}(G_2)_{(2)}$ is also an isomorphism.

By Proposition~\ref{proposition:6.3}, there is an integer $a$ together with a map $g$ such that $h_4\circ \phi\simeq a \cdot i'+g \circ p'$. The induced homomorphism $i'_*\colon \pi_3(G_{2,11})\to \pi_3(G_2)$ is an isomorphism. 
The induced homomorphism $({g \circ p'})_*\colon \pi_3(G_{2,11})\to \pi_{3}(G_2)$ is zero since it factors through $\pi_3(G_{2,11}/G_{2,3})\cong \{0\}$.
Therefore, the image of the induced homomorphism $(h_4\circ \phi)_*\colon \pi_{3}(G_{2,11})\to \pi_{3}(G_2)$ is $(a)\subset \mathbb{Z}\cong \pi_{3}(G_2)$.
On the other hand, by Proposition~\ref{proposition:6.5}, we have
\[
21\cdot 3 \cdot {\partial_4}_*(h_4 \circ \phi)\simeq a \cdot 21\cdot 3 \cdot {\partial_4}_*(i'). \]
Since the composition $\partial_4 \circ h_4$ is null homotopic, $h_4 \circ \phi$ is in the kernel of the induced homomorphism ${21\cdot 3 \cdot {\partial_4}_*}$. Therefore, by Proposition~\ref{proposition:6.4}, we have $a\equiv 0\mod(2)$, and the image of the induced homomorphism $(h_4 \circ \phi)_*\colon \pi_3(G_{2,11})_{(2)} \to \pi_{3}(G_2)_{(2)}$ is $(a)\subset (2)\subset \mathbb{Z}_{(2)}\cong \pi_{3}(G_2)_{(2)}$. This implies that the induced homomorphism $(h_4\circ \phi)_*\colon \pi_{3}(G_{2,11})_{(2)}\to \pi_{3}(G_2)_{(2)}$ is not an isomorphism. It is a contradiction.
\end{proof}





\end{document}